\documentclass[a4paper,11pt]{article}
\usepackage{graphicx,xcolor,textpos}
\usepackage{helvet}

\usepackage{amsfonts}
\usepackage{amssymb}
\usepackage{amsmath}
\usepackage{amsthm}
\usepackage{graphicx}
\usepackage{subcaption}
\usepackage{multicol}
\usepackage{siunitx}
\usepackage{url}
\usepackage[utf8]{inputenc}
\usepackage[nottoc]{tocbibind}
\usepackage{tikz}
\usepackage[symbol]{footmisc}
\usepackage{xcolor}
\usepackage{mleftright}
\usepackage{verbatim}
\usepackage[ruled,vlined]{algorithm2e}
\usepackage[noend]{algpseudocode}
\usepackage{lipsum}
\usepackage{mathrsfs}
\usepackage{tikz-cd}
\usepackage{hyperref}
\hypersetup{
	colorlinks=true,
	linkcolor=blue,
	filecolor=magenta,      
	urlcolor=cyan,
}
\usepackage{tabularx}
\usepackage{multirow}
\usepackage{longtable}
\usepackage{emptypage}
\usepackage{quiver}
\usepackage{dsfont}
\usepackage[capitalise,noabbrev]{cleveref}
\usepackage[toc,page]{appendix}
\usepackage[colorinlistoftodos]{todonotes}

\usepackage{a4wide}
\newtheorem{lemma}{Lemma}
\newtheorem{theorem}[lemma]{Theorem}
\newtheorem*{theorem*}{Theorem}
\newtheorem{definition}[lemma]{Definition}
\newtheorem*{definition*}{Definition}
\newtheorem{proposition}[lemma]{Proposition}
\newtheorem*{proposition*}{Proposition}

\newtheorem{corollary}[lemma]{Corollary}
\newtheorem*{corollary*}{Corollary}
\newtheorem{remark}[lemma]{Remark}

\newtheorem{question}{Question}
\newtheorem{conjecture}{Conjecture}

\crefalias{result}{theorem}

\newtheorem{innercustomthm}{Theorem}
\newenvironment{customthm}[1]
{\renewcommand\theinnercustomthm{#1}\innercustomthm}
{\endinnercustomthm}

\newcommand{\QQ}{\mathbb{Q}}
\newcommand{\RR}{\mathbb{R}}

\newcommand{\ZZ}{\mathbb{Z}}

\newcommand{\NN}{\mathbb{N}}

\newcommand{\id}{\mathds{1}}

\newcommand{\inv}{^{-1}}

\newcommand{\abs}[1]{\left|#1\right|}
\newcommand{\norm}[1]{\left|\left|#1\right|\right|}
\newcommand{\set}[1]{\left\{#1\right\}}

\newcommand{\mf}[1]{\mathfrak{#1}}
\newcommand{\mc}[1]{\mathcal{#1}}

\newcommand{\ggd}{\gcd}

\newcommand{\without}[1]{\backslash\{#1\}} 

\newcommand{\normal}{\vartriangleleft}

\newcommand{\range}[1]{\{1,2,\cdots,#1\}}

\newcommand{\floor}[1]{\left\lfloor #1\right\rfloor}

\newcommand{\gr}{\mathrm{Gr}}

\author{Lukas Vandeputte\footnote{The author  kindly acknowledges the support by the group of science Engineering and Technology at KU Leuven Campus Kulak.}}
\title{Twisted Conjugacy Growth of the Generalised Heisenberg Groups}
\date{\today}
\begin{document}
	\maketitle
	
	\begin{abstract}
		We determine the twisted conjugacy growth function for automorphisms on generalised Heisenberg groups. In particular we demonstrate for these groups that up to a natural equivalence this function is either given by $n^k$ or $n^k\ln(n)$ for some constant $k$. We give a precise description for which of these functions we obtain in terms of dimensions and ranks on the upper and lower central series. In particular we obtain that the twisted conjugacy growth is always bounded by the conjugacy growth.
	\end{abstract}
	\section*{Introduction}
	Let $G$ be a group, finitely generated by some set $S$ that is closed under inverses and that contains the identity element. The \textbf{conjugacy growth function} ${\gr}_{\sim G}^S:\NN_{>0}\rightarrow\NN_{>0}$ is an increasing function of integers where ${\gr}_{\sim G}^S(n)$ counts the number of distinct conjugacy classes of $G$, intersecting the closed ball $S^n$. This function was first introduced by Babenko in \cite{Babenko1988closed} due to its relationship with geodesics in Riemanian manifolds, but has since been studied in its own right. 
	This function can behave quite differently, from the normal growth function (which counts the number of elements in the closed balls $S^n$). Indeed, Osin \cite{Osin2010small} showed that groups with exponential growth can have bounded conjugacy growth. However, in a lot of cases, there is a closer correspondence between the $2$. For instance Breuillard and de Cornulier\cite{MR2777001} demonstrated that for solvable groups, either both the growth and conjugacy growth are exponential, or they are both at most polynomial, in which case the group must be virtually-nilpotent following Gromov's polynomial growth theorem \cite{Gromov1981groups}. Similarly Breuillard, de Cornulier, Lubotzky and Chen \cite{MR3021813} showed that the same result holds in the class of finitely generated linear groups.
	The above suggests that the conjugacy growth of nilpotent groups is of interest. Nevertheless the exact conjugacy growth of nilpotent groups is only known in rare cases such as the \textbf{generalised Heisenberg group}, finitely generated $2$-step nilpotent groups with infinite cyclic derived subgroup, by Evetts\cite{evetts2023conjugacy}.
	
	In the same paper, Evetts introduced the related twisted conjugacy growth function.
	Given $\psi$ an automorphism on $G$, we call $x,y\in G$ \textbf{$\psi$-conjugate} if there exists some $g\in G$ such that $gx\psi(g)\inv=y$, in this case we denote $x\sim_\psi y$. The \textbf{$\psi$-twisted conjugacy classes} are the equivalence classes of this relation and we denote the twisted conjugacy class of $x$ by $[x]_\psi$.
	The \textbf{$\psi$-twisted conjugacy growth function} is then the function ${\gr}_{\psi, G}^{S}:\NN\rightarrow\NN$ that maps $n$ to the number of $\psi$-conjugacy classes with a representative in $S^n$. Note that this function, like the conjugacy growth function, depends on the choice of generating set. This dependence is only slight and can be mitigated if we look only at the asymptotic behaviour.
	
	Originally, the twisted conjugacy growth function was introduced to study the untwisted variant. In particular, if $\gr_{\psi,H}$ is for any automorphism $\psi$ bounded by the same function $f$, then $\gr_{\tilde G}$ is also bounded by $f$ for any finite extension $G$ of $H$. 
	
	In more recent work, the twisted conjugacy growth has been studied on its own. More specifically in \cite{dekimpe2025twisted}, Dekimpe and Lathouwers computed the twisted conjugacy growth functions on finitely generated virtually abelian groups. 
	In this paper we continue this development by generalising the work from Evetts \cite{evetts2023conjugacy} on generalised Heisenberg groups to the setting of twisted conjugacy.
	For these groups we derive the following result:
	\begin{customthm}{A}\label{prop:maintheorem}
		Let $G$ be a generalised Heisenberg group. Let $\psi:G\rightarrow G$ be an automorphism and let $\overline{\psi}$ be the induced morphism of $\psi$ on the abelianization of $G$, then $${\gr}_{\psi, G}^{S}(n)\simeq n^{\dim(G/[G,G])-\mathrm{rank}(\id-\overline\psi)}d(n)$$ where $d$ is given by one of $n^2,n,\ln(n)$ or $1$.
	\end{customthm}
	For more precise conditions on the function $d$ we refer the reader to \cref{prop:completemaintheorem} where we give the expression of $d$ in terms of ranks on and dimensions of quotients of the upper and lower central series.
	This more precise formulation allows us to relate the twisted conjugacy growth that we computed to the conjugacy growth functions as computed in \cite{evetts2023conjugacy}.
	\begin{corollary}
		Let $G$ be a finitely generated group with infinite cyclic derived subgroup, let $\psi$ be any automorphism of $G$, then $$
		\gr_{\psi,G}(n)\prec \gr_{\sim G}(n).
		$$
	\end{corollary}
	
	Notice that by Dekimpe and Lathouwers\cite{dekimpe2025twisted} this is also the case in the class of virtually abelian groups. This leads to the following question:
	\begin{question}
		What finitely generated groups $G$ have the property that for any automorphism $\psi$ on $G$, $$
		{\gr}_{\psi,G}^S\prec {\gr}_{\sim G}^S?
		$$
	\end{question}
	Notice that if $G$ is such a group, then $G$ and any finite extension $H$ of $G$ have the same conjugacy growth function.
	Hull and Osin clearly provide a non-example to this in \cite[Theorem 1.3]{MR3010062}.
	On the other hand, the earlier mentioned results from \cite{MR2777001} and \cite{MR3021813}, imply that the above holds for all solvable and linear groups that are not virtually nilpotent.

	Lastly, as a related concept to the twisted conjugacy growth function, we also consider the \textbf{twisted conjugacy growth series}. This is the formal power series 
	$\sum {\gr}_{\psi,G}^S(n)x^n $
	where the coefficients are given by the values of the twisted conjugacy growth function.
	In \cite{MR4102998} Ciobanu, Evetts and Ho conjectured that the conjugacy growth series of a finitely presented group is rational if and only if it is virtually abelian. As a closely related conjecture we have the following:
	\begin{conjecture}\label{conj:twistedtrancendental}
		A finitely presented group $G$ is virtually abelian, if and only if all its twisted conjugacy growth functions are rational.
	\end{conjecture}
	
	Recent work of Evetts and Lathouwers\cite{evetts2024twisted} suggests that the latter conjecture is in fact a weakening of the former. As a consequence of our \cref{prop:completemaintheorem} we obtain the following:
	\begin{corollary}
		Let $G$ be a generalised Heisenberg group, then there exists an automorphism $\psi$, such that the twisted conjugacy growth series with respect to that automorphism is transcendental for any generating set.
	\end{corollary}
	The above result gives evidence that \cref{conj:twistedtrancendental} holds, at least for the class of nilpotent groups.
	
	We end the introduction by giving a rundown of the remainder of this paper. In \cref{sec:prelim} we start by introducing some necessary definitions together with some basic results on twisted conjugacy growth.
	In \cref{sec:description} we build up to the complete statement of our main result, as well as its consequences. A first case of the main theorem is proven in \cref{sec:nonDegen}. In \cref{sec:classes} we describe the twisted conjugacy classes in the other case. After \cref{sec:numbertheory}, which mainly estimates some power series, we complete the proof of the second case in \cref{sec:degen}.

	\section{Preliminaries}\label{sec:prelim}\label{sec:twisted}
	\setcounter{lemma}{0}
	\numberwithin{lemma}{section}
	To define the twisted conjugacy growth, 
	first we introduce the word norm on a group.
	\begin{definition}
		Let $S$ be a finite generating set of the group $G$, for $g\in G$ define then the \textbf{word norm} 
		$$\norm{g}_S=\begin{cases}\begin{matrix}
				0 &\text{if } g=e_G\\
				\min\set{n\in\NN\mid g=s_1s_2\cdots s_n\text{ for some } s_i\in S\cup S\inv} &\text{otherwise}.
			\end{matrix}
		\end{cases}$$
	\end{definition}
	The word norm depends on the choice of generating set, but only by at most a constant.
	
	\begin{lemma}
		\label{prop:normAreEquiv}
		Let $G$ be a finitely generated group and let $S$ and $T$ be two finite generating subsets of $G$. Then there exists a constant $C$ such that for every group element $g$ the following holds: $$\norm{g}_T\leq C\norm{g}_S. $$
	\end{lemma}
	For a proof, see for instance \cite{loh2017geometric}.
	
	Word norms are relatively well behaved with respect to quotients.
	\begin{lemma}\label{prop:normPreservingLiftExt:lemma}
		Let $\pi:G\rightarrow H$ be a surjective morphism of groups.
		Let $S$ be a generating set of $G$ and let $\pi(S)=\set{\pi(s)\mid s\in S}$. For any $g\in G$, $\norm{g}_S\geq\norm{\pi(g)}_{\pi(S)}$ and furthermore, there exists some $g'$ such that $\pi(g')=\pi(g)$ and such that $\norm{g'}_S=\norm{\pi(g)}_{\pi(S)}$.
	\end{lemma}
	\begin{proof}
		This is immediate.
	\end{proof}

	We reformulate our previous definition of the twisted conjugacy growth in terms of the word norm.
	\begin{definition}
		Let $G$ be a group, finitely generated by $S$ and let $\psi:G\rightarrow G$ be an automorphism, then we define the \textbf{twisted conjugacy growth function} as $$
		{\gr}_{\psi,G}^S(n)=\#{\set{[g]_\psi\big\vert g\in G, \norm g_S\leq n}}.
		$$
	\end{definition}
	Notice that the above definition depends on the generating set $S$. To mitigate this we introduce the following equivalence of functions.
	\begin{definition}
		Let $f,g:\NN\rightarrow\RR_{>0}$ be non-decreasing functions. Then we say $f\prec g$ if there exists some natural constant $c$ such that $f(n)\leq cg(cn)$ for all $n\in \NN_0$. Furthermore if $f\prec g$ and $g\prec f$, then we say $f\simeq g$
	\end{definition}
	\begin{remark}
		Throughout the paper, we will often compare functions which also take negative values such as $\ln(n)$, If $f,g:\NN\rightarrow\RR$ are non-decreasing functions, then we denote $f(n)\prec g(n)$ if $\max(1,f(n))\prec\max(1,g(n)).$
	\end{remark}
	For $S$ and $S'$ two generating sets of the same group $G$, and for some automorphism $\psi$, we now have ${\gr}_{\psi,G}^S\simeq {\gr}_{\psi,G}^{S'}$. This allows us to unambiguously speak of ${\gr}_{\psi,G}$ without having to specify the generating set.
	
	In general, the twisted conjugacy growth function need not be well behaved under morphisms.
	When the morphism is surjective however we have the following.

	\begin{proposition}\label{prop:finiteCenterReduction}
		Let $G$ be a group and let $N\normal G$ be a normal subgroup and let $\psi:G\rightarrow G$ be an automorphism such that $\psi(N)=N$. Denote with $\overline\psi$ the induced automorphism of $\psi$ on $\frac{G}{N}$, then ${\gr}_{\psi,G}\succ {\gr}_{\overline\psi,\frac{G}{N}}$. Suppose furthermore that there exists some constant $C$ such that for any $g\in G$ the set $\set{[ng]_\psi\mid n\in N}$ is finite of size at most $C$, then $$
		{\gr}_{\psi,G}\simeq {\gr}_{\overline\psi,\frac{G}{N}}.
		$$
	\end{proposition}
	\begin{proof}
		Notice that the above statement does not depend on the choice of generating set. Fix $S$ a generating set of $G$, and let $\pi(S)=\set{\pi(s)\mid s\in S}$.
		First we will demonstrate that ${\gr}_{\overline\psi,G/N}^{\pi(S)}(n)\leq {\gr}_{\psi,G}^{S}(n)$. Let $[g]$ be a $\overline\psi$-conjugacy class, such that $\norm{g}_{\pi(S)}\leq n$ for some $g\in [g]$. By Lemma \ref{prop:normPreservingLiftExt:lemma} there exists some element $g'\in G$, such that $\pi(g')=g$ and such that $\norm{g'}_S\leq n$.
		By the axiom of choice we thus have a function $\chi$, mapping every twisted conjugacy class of $G/N$ with a representative $g$ of norm at most $n$, to a twisted conjugacy class of $G$ with a representative $\chi(g)=g'$ of norm at most $n$. This map is injective, indeed, a left inverse of this map is given by $\pi$. We thus have that ${\gr}_{\overline\psi,G/N}^{\pi(S)}(n)\leq {\gr}_{\psi,G}^{S}(n)$.
		
		For the second part, the map $\pi$ induces a set map $\tilde\pi$ from the twisted conjugacy classes of $G$ to those of $G/N$. As $[\pi(a)]_{\overline\psi}=\pi([a]_{\psi})$, this map is at most $C$-to-$1$.
		Notice also that by \cref{prop:normPreservingLiftExt:lemma}, $\pi$ does not increase the word norm. In particular, $\tilde\pi$ maps twisted conjugacy classes with a representative $g$ such that $\norm{g}_S\leq n$, to a twisted conjugacy class with a representative $\pi(g)$ such that $\norm{\pi(g)}_S\leq n$. It follows that there are at most $C$ times as many twisted conjugacy classes in $G$, with a representative of norm at most $n$, as there are twisted conjugacy classes in $G/N$ with a representative of norm at most $n$.
		
		By the above two arguments it follows that $${\gr}_{\psi,G}^{S}(n)\geq {\gr}_{\overline\psi,G/N}^{\pi(S)}(n)\geq \frac{1}{C}{\gr}_{\psi,G}^{S}(n).$$
	\end{proof}
	In the special case where $N$ above is finite, we then obtain the following result.
	\begin{corollary}
		If $G$ is as a group finite-by-$H$, and $\psi$ is an automorphism inducing the automorphism $\overline\psi$ on $H$, then $$
		{\gr}_{\psi,G}(n)\simeq {\gr}_{\overline\psi,H}.
		$$
	\end{corollary}

	To end this subsection, we give the twisted conjugacy growth function for the torsion free finitely generated abelian groups.
	\begin{theorem}[\cite{dekimpe2025twisted}]\label{prop:freeAbelian}
		Let $A$ be a finitely generated free abelian group of dimension $d$ and let $\psi$ be an endomorphism of $A$. Then $$\gr_{\psi,G}(n)\simeq n^{d-\mathrm{rank}(\id-\psi)}. $$
	\end{theorem}
	\section{Main result}\label{sec:description}
	\subsection{Describing the groups}\label{sec:BCH}
	Let $G$ be a group and let $g,h\in G$, then we denote the group commutator $[g,h]=ghg\inv h\inv$.
	Let the \textbf{derived subgroup} $[G,G]$ be the group generated by $\{[g,h]\mid g,h\in G\}$. Let the \textbf{centre} be the subgroup $Z(G)=\{g\in G\mid \forall h\in G:[g,h]=e\}$. We call a group \textbf{$2$-step nilpotent} if $[G,G]\subset Z(G)$. We denote $G_2$ the isolator $\sqrt{[G,G]}$. 
	
	If furthermore $G_2=\ZZ$, then we call $G$ a \textbf{generalised Heisenberg group}.
	
	If $G$ is torsion free $2$-step nilpotent, then $Z(G),[G,G],G_2,\frac{G}{G_2}$ and $\frac{G}{Z(G)}$ are all torsion free abelian groups. In this case we define the \textbf{Hirsch length} $h(G)=\dim(G_2)+\dim(\frac{G}{G_2}).$
	
	Let $(\mf g,+,[\_,\_])$ be a lie-algebra. We say $\mf g$ is \textbf{$2$-step nilpotent} if $[\mf g,[\mf g,\mf g]]=0$. There is a close connection between torsion free nilpotent groups and nilpotent lie-algebras. For simplicity, we restrict ourselves to the $2$-step nilpotent case. See \cite[Chapter 6]{segal1983polycyclic} for a more in depth treatment.
	
	Let $(\mf g,+,[\_,\_])$ be a two step nilpotent lie-algebra over $\QQ$. Then we define the binary operator $*$ as$$
	*:\mf g\times\mf g\rightarrow\mf g:a,b\mapsto a*b=a+b+\frac{1}{2}[a,b].
	$$
	One shows that $(\mf g,*)$ is a torsion-free two-step nilpotent group. Furthermore, the group commutator of $\mf g,*$ is given by the lie-bracket $[\_,\_]$.
	
	Suppose we have a finitely-generated torsion-free two-step nilpotent group $(G,*)$. Then there exists some finite-dimensional two-step nilpotent Lie-algebra $(\mf g,+,[\_,\_])$, together with a group embedding $\log:(G,*)\rightarrow(\mf g,*)$ such that $\mf g=\log G\otimes\QQ$. 	
	The last condition implies that the dimension of $\mf g$ is equal to the Hirsch-length of $G$.

	Furthermore, if $\psi$ is an automorphism of $G$. Then there exists a unique Lie-algebra automorphism $\psi_{\mf g}$ on $\mf g$ such that the diagram commutes.
	% https://q.uiver.app/#q=WzAsNCxbMCwyLCJcXG1mIGciXSxbMiwyLCJcXG1mIGciXSxbMCwwLCJHIl0sWzIsMCwiRyJdLFsyLDMsIlxcdmFycGhpIl0sWzAsMSwiXFx2YXJwaGlfe1xcbWYgZ30iXSxbMiwwLCJcXGxvZyIsMV0sWzMsMSwiXFxsb2ciLDFdXQ==
	\[\begin{tikzcd}
		G && G \\
		\\
		{\mf g} && {\mf g}
		\arrow["\psi", from=1-1, to=1-3]
		\arrow["\log"{description}, from=1-1, to=3-1]
		\arrow["\log"{description}, from=1-3, to=3-3]
		\arrow["{\psi_{\mf g}}", from=3-1, to=3-3]
	\end{tikzcd}\]
	
	Proofs of tje above statements all derive from \cite{segal1983polycyclic}
	Let $G$ be a generalised Heisenberg group, then the commutator induces a bilinear skew-symmetric map: $\omega:G/G_2\times G/G_2\rightarrow G_2\cong\ZZ:gG_2,hG_2\mapsto [g,h]$. This map is bilinear and skew symmetric.
	
	Using $\omega$, we can construct a new group $H_\omega$.
	\begin{definition}
		Let $(H_\omega,*)$ be the group on the set $G/G_2\times G_2$ where $*$ is given by$$
		\begin{matrix}
		*&:&(G/G_2\times G_2)\times (G/G_2\times G_2)&\rightarrow& G/G_2\times G_2\\&:&((a,c)(a',c'))&\mapsto& (a+a',c+c'+\omega(a,a')).
		\end{matrix}$$
	\end{definition}
	
	This group $H_\omega$ is not isomorphic to $G$. However, $G$ can be embedded as an index $2$ subgroup into $H_\omega$.
	Indeed let $\mf g$ be the Lie-algebra associated to $G$, then $(H_\omega,*)$ is isomorphic to the subgroup $(\ZZ\log(G)+\frac{1}{2}\log(G_2),*)<(\mf g,*)$. Such an isomorphism can be realised by fixing $a_i$ a minimal generating set of $G$, $c$ a generator of $G_2$, and mapping $(a_iG_2)$ to $\log(a_i)$ and $(0,c)$ to $\log(c)$.
	
	This embedding is in such a manner that $G_2$ is of index $2$ in $(H_\omega)_2$ and such that the following diagram commutes:
	% https://q.uiver.app/#q=WzAsMyxbMCwwLCJHIl0sWzIsMCwiRy9aXFx0aW1lcyBaIl0sWzEsMiwiRy9aIl0sWzAsMSwiXFxjb25nIl0sWzAsMiwiXFxwaV9aIiwyXSxbMSwyLCJcXHBpXzEiXV0=
	\[\begin{tikzcd}
		G && {G/G_2\times G_2} \\
		\\
		& {G/G_2}
		\arrow[ hook,from=1-1, to=1-3]
		\arrow["{\pi_{G_2}}"', from=1-1, to=3-2]
		\arrow["{\pi_1}", from=1-3, to=3-2]
	\end{tikzcd}.\]
	
	If in the above construction, the commutator subgroup $[G,G]$ has odd index in $G_2$, then $H_\omega$ is the lattice hull of $G$, otherwise $G$ and $H_\omega$ are both lattice groups.
	
	The embedding of $G$ into $H_\omega$ is not unique, but we fix a single embedding in the remainder of the paper, and regard $G$ as a subgroup of $H_\omega$.

	\subsection{The automorphisms}
	Let $\psi$ be an automorphism of $G$. We have that $\psi$ extends uniquely to an automorphism $\psi_{\mf g}$ on $\mf g$. Furthermore, as Lie-ring automorphisms preserve addition and the commutator, $\psi$ restricts to an automorphism on $H_\omega$. As no confusion is possible, we also denote this automorphism with $\psi$.
	Using the linearity of $\psi_\mf g$, we obtain that this automorphism can be written as
	$$
	\psi(g,h)=(\psi_c(g),\psi\vert_{G_2}(h)+\psi'(g))
	$$
	where $\psi_c$ is the induced morphism of $\psi$ on $G/G_2$, $\psi\vert_{G_2}$ is the restriction of $\psi$ to $G_2$, and $\psi'$ is a linear map from $G/G_2$ to $G_2$. We determine which combinations of maps $\psi_c,\psi\vert_{G_2}$ and $\psi'$ are possible.
	
	As $G_2$ is infinite cyclic, we have that $\psi\vert_{G_2}$ is either the identity $\id$, or taking the opposite $-\id$. We also must have that $\psi$ preserves the commutator, that is for any $g,h\in G$ we have:$$
	\psi[g,h]=[\psi(g),\psi(h)]
	$$
	and thus from $g,h\in\frac{G}{G_2} $ that $$
	\psi\vert_{G_2}(\omega(g,h))=\omega(\psi_c(g),\psi_c(h)).
	$$
	Thus if $\psi\vert_{G_2}$ is the identity, then $\psi_c$ preserves $\omega$, and if $\psi\vert_{G_2}$ is taking the opposite, then $\psi_c$, changes the sign of $\omega$.

	We will need multiple variations of $\psi$ on different subgroups and quotients of $H_\omega$, for this we introduce the following notation.

	\begin{definition}
		Let $G$ be as above, let $\psi$ be an automorphism on $G$, denote then:\begin{itemize}
			\item $\psi_c$ the induced map on $G/G_2$, $d_c$ the dimension of $G/G_2$ and $r_c$ the rank of $\psi_c-\id$;
			\item $\psi_z$ the induced map on $G/Z(G)$, $d_z$ the dimension of $G/Z(G)$ and $r_z$ the rank of $\psi_z-\id$;
			\item $\psi_{z/c}$ the restriction of $\psi_c$ to $Z(G)/G_2$, $d_c$ the dimension of $G(Z)/G_2$ and $r_{z/c}$ the rank of $\psi_{z/c}-\id$.
		\end{itemize}
	\end{definition}
	
	The choice of the function $d$ from \cref{prop:maintheorem}, will be determined in large part by the following constant.
	\begin{definition}
		Given a group $G$ as above and an automorphism $\psi$ on $G$ as above, define then
		$$
		\mf d_\psi=(d_c-r_c)-(d_{z/c}-r_{z/c}).
		$$
	\end{definition}
	
	We are now ready to state \cref{prop:maintheorem} in full specificity.

	\begin{customthm}{B}\label{prop:completemaintheorem}
		Let $G$ be a generalised heisenberg group with automorphism $\psi$, then $$\gr_{\psi,G}(n)\simeq n^{d_c-r_c}d_{\psi}(n)$$ where $d$ is given by$$\begin{cases}
			\text{$d(n)=1$ if $\psi\vert_{G_2}=-1$ or if $\psi'(\ker(\psi_{z/c}-\id))$ is non-trivial;}\\
			\text{$d(n)=1$ otherwise if $\mf d_\psi\geq 3$;}\\
			\text{$d(n)=\log(n)$ otherwise if $\mf d_\psi=2$;}\\
			\text{$d(n)=n$ otherwise if $\mf d_\psi=1$;}\\
			\text{$d(n)=n^2$ otherwise if $\mf d_\psi=0$.}
		\end{cases}$$
	\end{customthm}
	\begin{remark}
		\sloppy
	In the above theorem we can reformulate the condition ``$\psi\vert_{G_2}=-1$ or ${\psi'(\ker(\psi_{z/c}-\id))}$ is non-trivial" as ``There exists some $g\in Z(G)$ such that $\psi_c(gG_2)=gG_2$ but $\psi(g)\neq g$."
	If these conditions hold, then we say we are in the \textbf{non-degenerate} case.
	\end{remark}
	As an immediate consequence of the main theorem, we have the following relation with the conjugacy growth.
	\begin{corollary}
		Let $G$ be a generalised Heisenberg group and let $\psi$ be an automorphism of $G$, then $$
		\gr_{\sim G}(n)\succ \gr_{\psi,G}(n).
		$$
	\end{corollary}
	\begin{proof}
		As $\gr_{\sim G}(n)=\gr_{\id G}(n)$, we have that $\gr_{\sim G}(n)\simeq n^{d_c}d_{\mf d_\id}(n)$ where $\mf d_\id$ is given by $d_c-d_{z/c}$. Let $r_c$ be the rank of $\id-\psi_c$, then $\mf d_\psi\geq \mf d_\id -r_c$. One checks from the definition of $d$ that this implies that $d_{\mf d_\psi}(n)\prec n^{r_c}d_{\mf d_\id}(n)$, or thus after multiplying both sides with $n^{d_c-r_c}$ that $$n^{d_c-r_c}d_{\mf d_\psi}(n)\prec n^{d_c}d_{\mf d_\id}(n).$$
	\end{proof}
	
	\begin{lemma}\label{prop:nonPolynomial}
		For every torsion free group $G$ with infinite cyclic commutator subgroup, there exists some automorphism $\psi$ such that $\gr_{\psi,G}\simeq n^k\ln(n)$ for some integer $k$.
	\end{lemma}
	\begin{proof}
		In the case where $d_z=2$, we have by \cite{evetts2023conjugacy} that $\gr_{\id,G}=\gr_{\sim G}\simeq n^2\ln(n)$. Assume thus for the remainder that $d_z\geq 4$.
		
		The centre $Z(G)$ is isolated and thus can we realise $G/G_2$ as $Z(G)\oplus G'$ for some $G'<G/G_2$. Modding out $Z(G)$ then gives an identification between $G'$ and $G/Z(G)$.
		We have that $\omega$ is non-degenerate on $G/Z(G)$ and thus a symplectic form on $G/Z(G)\otimes \QQ$. We may thus find a symplectic basis $\{p_1,q_1,p_2,q_2,\cdots p_m,q_m\}$. After rescaling these elements, we have that as $\ZZ$-modules $$\langle p_1,q_1,\cdots,p_m,q_m\rangle\supset G/Z(G)\supset M\langle p_1,q_1,\cdots,p_m,q_m\rangle$$ for some constant $M$.
		Consider then the linear map $\psi_z$ acting on $(p_1,q_1)$ and $(p_2,q_2)$ as the matrix$$
		\begin{pmatrix}
			1&M\\
			0&1\\
		\end{pmatrix}
		$$
		and on all other pairs $(p_i,q_i)$ as the matrix$$
		\begin{pmatrix}
			1-2M+2M^2&2M^2\\
			2M^2&1+2M+2M^2
		\end{pmatrix}.
		$$
		Notice that both of these matrices have determinant $1$ and thus preserve the symplectic form. Furthermore, $\psi_z$ restricts to a map on the $\ZZ$-module $G/Z(G)$: $\psi_z$ splits as $\id+M\tilde\psi$ for some $\tilde\psi$ with integer coefficients. The map $\id$ preserves $G/Z(G)$ and $M\tilde\psi$ maps elements of $G/Z(G)$ to  $M\langle p_1,q_1,\cdots,p_m,q_m\rangle\subset G/Z(G)$.
		We then lift $\psi_z$ to a map $\psi_c$ on $G/G_2$ by requiring that $\psi_c$ acts trivially on $Z(G)/G_2$, and as $\psi_z$ on $G'\cong G/Z(G)$. 
		
		Now we want to lift $\psi_c$ to a map $\psi:G\rightarrow G$ such that $\psi$ belongs to the degenerate case. 
		Let $\{a_1,a_2,\cdots a_{2m},\cdots, a_k\}$ be a basis of $G/G_2$ such that $a_{2m+1},\cdots,a_k$ is a basis of $Z(G)/G_2$. Let $b_i$ be such that $(a_i,b_i)\in G$ and $b'_i$ such that $(\psi_c(a_i),b_i')\in G$. Let $\psi':G/G_2\rightarrow \ZZ$ then be the unique linear map, mapping $a_i$ to $0$ if $b_i\cong b_i'\mod 2$ and to $1$ otherwise. Notice that $\psi'(Z(G)/G_2)=0$. 
		
		Let $\psi$ then be given by $\psi:H_\omega\rightarrow H_\omega:(a,b)\mapsto (\psi_c(a),\psi'(a)+b)$. Using that $\psi_z$ preserves $\omega$, one checks that this is indeed a group morphism on $H_\omega$. By construction of $\psi'$, this map restricts to a map $\psi:G\rightarrow G$. Furthermore, we claim that $d_\psi(n)=\log(n)$. Indeed as mentioned before, we have that $\psi'(Z(G)/G_2)=0$. Furthermore, by construction of $\psi_c$, we have that $d_c-r_c=2+d_{z/c}$ and $r_{z/c}=0$. As $\psi_c$ preserves the decomposition $Z(G)\oplus G'$, we have $d_c=d_z+d_{z/c}$ and $r_c=r_z+r_{z/c}$. It follows that $\mf d_\psi=d_z-r_z$ which by construction is equal to $2$ and thus we have that $\gr_{\psi,G}(n)\simeq n^k\ln(n)$.
	\end{proof}
		From \cite[Theorem 2.16]{evetts2023conjugacy}, we have that if $\gr_{\psi,G}$ does not grow polynomially, then the growth series is transcendental. By \cref{prop:nonPolynomial} we thus have:
	\begin{corollary}
		For every torsion free group $G$ with infinite cyclic commutator subgroup, there exists some automorphism $\psi$ such that the twisted conjugacy growth series of $G$ and $\psi$ is transcendental.
	\end{corollary}
	\section{Proof of the non-degenerate case}\label{sec:nonDegen}
	First we restrict ourselves to the case where either $\psi\vert_{G_2}=-1$ or where $\psi'(\ker(\psi_{z/c}-\id))$ is non-trivial.
	The idea in both cases is to use \cref{prop:finiteCenterReduction} to reduce to the abelian case.
	\begin{lemma}\label{prop:negsign}
		Let $\psi$ be an automorphism of $G$, such that $\psi\vert_{G_2}$ flips the sign.
		Then for any $(a,b)\in G$ and any $n\in\ZZ$, we have that $(a,b+4n)\in G$ and $(a,b)\sim_\psi(a,b+4n)$ in $G$.
	\end{lemma}
		\begin{proof}
			Notice that $(0,2n)\in G$. We have\begin{align*}
			(0,2n)(a,b)\psi(0,2n)\inv&=(0,2n)(a,b)(0,2n)\\
			&=(a,b+4n).
		\end{align*}
		The result follows immediately.
	\end{proof}
	For the second case we do something similar.
	\begin{lemma}\label{prop:centershift}
	Let $\psi$ be an automorphism of $G$ such that there exists some $g\in \ker(\psi_{z/c}-\id)$ satisfying $\psi'(g)\neq 0$,
	then for any $(a,b)\in G$ and any $n\in \ZZ$, $$(a,b)\sim_\psi(a,b+2n\psi'(g)).$$
	\end{lemma}
	\begin{proof}
		Let $(-g,c)\in G$, then $(-2g,2c)\in G$ and thus $(-2g,0)\in G$.
		Furthermore,
		 \begin{align*}
			(-2ng,0)(a,b)\psi(-2ng,0)\inv&=(-2ng,0)(a,b)(2ng,2n\psi'(g))\\
			&=(a,b+2n\psi'(g)).
		\end{align*}
		The first equality holds because $\psi_c(g)=g$ and the second because $\omega(g,\_)$ vanishes.
	\end{proof}
	Notice that the above also holds when $\psi'(g)=0$, but then the statement becomes trivial.
	
	We now obtain the first case of \cref{prop:completemaintheorem}:
	\begin{customthm}{B.1}
		Let $G$ be as before. Suppose that either $\psi\vert_{G_2}=-1$ or that $\psi'$ is non-trivial on $\ker(\psi_{z/c}-\id)$. Then $$
		\gr_{\psi,G}\simeq n^{\dim(G/G_2)-r_c}
		$$
	\end{customthm}
	\begin{proof}
		This is immediate by combining either \cref{prop:negsign} or \cref{prop:centershift} with \cref{prop:finiteCenterReduction} and \cref{prop:freeAbelian}.
	\end{proof}

	\section{Twisted conjugacy classes of the degenerate case}\label{sec:classes}
	For the remainder of this paper we will assume that $\psi$ is degenerate, that is $\psi\mid_{G_2}=\id$ and for $g\in Z(G)/G_2$, if $\psi_c(g)=g$, then $\psi'(g)=0$.
	
	In this subsection we describe the $\psi$-conjugacy classes of $G$.

	In the ``negative" case, twisted conjugation with elements of $G_2$ played an important role, here however this will not be the case. The reason being the following lemma:
	\begin{lemma}\label{prop:positiveDetCenterIrelivantExt:lemma}
		$\psi$-Conjugating with $(a,c)$ or with $(a,0)$ yields the same result.
	\end{lemma}
	\begin{proof}
		Notice that $(a,c)=(a,0)(0,c)$. We thus need to demonstrate that $\psi$-conjugation with $(0,c)$ does not change the result. Indeed $\psi((0,c))=(0,c)$ and thus we have $$
		(0,c)(a',c')\psi((0,c))\inv=(a',c'+c-c).
		$$
	\end{proof}
	In this case, the $\psi$-conjugation action thus only depends on the abelianization.
	Furthermore, we have for $(a,c)\in H_\omega$ that either $(a,c)\in G$ or $(a,c+1)\in G$. We thus have that $2$ elements of $G$ are $\psi$-conjugate in $G$ if and only if they are in $H_\omega$.
	
	One piece of setup we will often need is the following.
	\begin{definition}
		Let $M$ be a free $\ZZ$-module and let $I$ and $J$ be submodules of $M$, then we call $I$ and $J$ transversal if\begin{itemize}
			\item $I+J$ is finite index in $M$;\\
			\item $I\cap J={\{0\}}$.
		\end{itemize}
		In this case we say that $J$ is a transversal of $I$.
		If furthermore $\mc A$ is a finite set containing $0$ such that $I+J+\mf A=M$, then we call $(I,J,\mf A)_M$ a transversality triple.
	\end{definition}
	Let $I$ be a submodule of a finitely generated free $\ZZ$-module, then there always exists some $J$ and $\mf A$ such that $(I,J,\mf A)_M$ is a transversality triple.

	\begin{lemma}\label{prop:conjToCokernel:lemma}

		Let $M$ be a finite dimensional free $\ZZ$-module with finite generating set $S$. Let $\psi$ be an automorphism on $M$, let $I$ be the image of $\psi-\id$ and let $(I,J,\mf A)_M$ be a transversality triple. Then we can decompose $v\in M$ as $$
		v=\left(\psi(w)-w\right) +j + a
		$$
		for some $w\in M$, $j\in J$ and $a\in\mf A$ where $\norm{w}_S\leq C_1\norm{v}_S$ and $\norm{j}_S\leq C_2\norm{v}_S$ for some $C_1,C_2>0$ only depending on $(I,J,\mf A)_M$.
	\end{lemma}
	\begin{proof}
		Passing to a subset of $\mf A$ if necessary, we may assume that $\mf A$ contains exactly one representative of each of the cosets of $\frac{M}{I+J}$. As $\mf A$ is a finite set, the norm of its element is bounded above, denote this bound $\norm{\mf A}_S$.
		As $I\cap J =0$, there exits a unique pair of linear maps $\pi_I:I+J\rightarrow I,\pi_J:I+J\rightarrow J$ satisfying $v=\pi_I(v)+\pi_J(v)$ for all $v\in I+J$.  Let $\norm{\pi_I}_S$ and $\norm{\pi_J}_S$ denote their operator norm, that is $$\norm{\pi_I}_S=\sup_{g\in M\without 0}\frac{\norm{\pi_I(g)}_S}{\norm{g}_S}$$ and $\norm{\pi_j}_S$ is defined analogously.
		These maps are linear maps on finite dimensional spaces thus these norms are well defined positive real numbers.
		
		Let $K$ be the kernel of $\psi-\id$. As $K$ is the kernel of a linear map with a torsion free codomain, $K$ must be radical. It follows that there exists some space $K^\perp<M$ such that $M=K\oplus K^\perp$. As $\psi-\id$ is a map from $M$ to $M$ with $I$ as range, and $K$ as kernel, $\psi-\id$ induces a linear isomorphism between $\frac{M}{K}=K^\perp$ and $I$. Denote $(\psi-\id)\inv:I\rightarrow K^\perp$ this inverse and denote $\norm{(\psi-\id)\inv}_S$ its norm.
		
		By definition of $\mf A$, there exits a unique element $a_v\in\mf A$ such that $v-a_v\in I+J$.
		Let now $j=\pi_J(v-a_v)$ and let $w=(\psi-\id)\inv(\pi_I(v-a_v))$. For $C_1=\norm{\pi_i}_S\norm{(\psi-\id)\inv}_S(1+\norm{\mf A}_S)$ and $C_2=\norm{\pi_J}_S(1+\norm{\mf A}_S)$, we now have that $\norm{j}_S\leq C_2\norm{v}_S$ and $\norm{w}_S\leq C_1\norm{v}_S.$
		Furthermore, $(\psi-\id)(w)+j=\pi_I(v-a_v)+\pi_J(v-a_v)=v-a_v$ and thus, rearranging the terms, we obtain$$
		v=\psi(w)-w+j+a_v.
		$$
		
	\end{proof}
	If we use the above lemma on the abelianization $G/G'$, we obtain a similar result on $G$ itself.
	\begin{lemma}\label{prop:conjToCokernelStrong:lemma}
		Let $S$ be a generating set of $G$, let $I$ be the image of $(\psi_c-\id)$ in $G/G_2$ and let $(I,J,\mf A)_{G/G_2}$ be a transversality triple. Then there exists some constant $C$, such that every element $v\in G$ is $\psi$-conjugate to an element $v'$ such that $\norm{v'}_S\leq C\norm{v}_S$ and such that $v'G_2\in G/G_2$ is of the form $a+j$ for some $a\in\mf A$ and $j\in J$.
	\end{lemma}
	\begin{proof}
		Let $\pi(S)=\set{\pi(s)\mid s\in S}$ be a generating set on $G/G'$. We may use Lemma \ref{prop:conjToCokernel:lemma}, to find some constant $C_0$ such that for any $v\in G/G_2$, there exists some $w\in G/G_2$ such that $\norm{w}_{\pi(S)}\leq C\norm{v}_{\pi(S)}$ and such that $w+v-\psi_c(w)$ is of the form $j+a$ with $j\in J$ and $a\in\mf A$.
		
		For $v'\in G$, we have that $\pi(v')\in G/G_2$ and thus can we find $w'$ in $G/G_2$ such that the above holds. By \cref{prop:normPreservingLiftExt:lemma}, there exists some $\overline{w'}\in G$ such that $\pi(\overline{w'})=w'$ and such that $\norm{\overline{w'}}_{S}=\norm{w'}_{\pi(S)}$. As $\pi$ is a group-morphism, we then have that $\pi(\overline{w'}v'\psi(\overline{w'}))$ is of the form $j+a$. With $j\in J$ and $a\in\mf A$. Furthermore, using the triangle inequality, $\overline{w'}v'\psi(\overline{w'})$ is of norm at most $\norm{v}_S(1+C_0+C_0\norm{\psi}_S)$.
	\end{proof}
	The above lemma tells that every element is $\psi$-conjugate to some other element $g$ with $gG_2\in J+\mf A$. The following lemma states that this representation is essentially unique up to $G_2$.
	\begin{lemma}\label{prop:conjDifferByI:lemma}
		Let $v_1,v_2$ be two $\psi$-conjugate elements of $G$. Then $\pi(v_1)-\pi(v_2)$ lies in the image of $\psi_c-\id$.
	\end{lemma}
	\begin{proof}
		If $v_1$ and $v_2$ are $\psi$-conjugate, then there exists some $w$ such that $wv_1\psi(w)\inv=v_2$. In particular, after applying $\pi$, we obtain$$
		\pi(w)+\pi(v_1)-\psi_c\pi(w)=\pi(v_2)
		$$
		which can be rewritten as $$\pi(v_1)-\pi(v_2)=(\psi_c-\id)(-\pi(w).$$
	\end{proof}

	As a basic strategy we will attempt to count for each element $v$ in the set $J+\mf A$ how many $\psi$-conjugacy classes there are with a representative $v'\in G$ such that $\pi(v')=v$.
	To do this, we give a condition for when two elements with the same abelianization are $\psi$-conjugate.
	\begin{lemma}\label{prop:centerDifferenceInOmega:lemma}
		Let $(a,b),(a,b')\in G<H_\omega$ be two elements, then they are $\psi$-conjugate if and only if there exists some $v\in G/G'$ such that $\psi_c(v)=v$ and such that $2\omega(v,a)-\psi'(v)=b'-b$.
	\end{lemma}
	\begin{proof}
		Let $v=(v_1,v_2)\in G$ be arbitrary, then $v(a,b)\psi(v)\inv$ is given by$$
		(a+v_1-{\psi_c}(v_1),b+\omega(v_1,a)+\omega(v_1,-{\psi_c}(v_1))+\omega(a,-{\psi_c}(v_1))-\psi'(v_1)).
		$$
		For the above to be equal to $(a,b')$, we need that $v_1=\psi(v_1)$. In this case the above reduces to$$
		(a,b+2\omega(v_1,a)-\psi'(v_1)).
		$$
		as had to be shown.
	\end{proof}
	
	As the above suggests, the maps $2\omega(v,\_)-\psi'(v):G/G'\rightarrow G'$ play a crucial role.
	
	\begin{definition}
	Let $I$ be the image of $\psi_z-\id$ and let $J<G/Z(G)$ be transversal to $I$. Let $K$ be the kernel of $\psi_z-\id$. Define then $$\omega_k:J\rightarrow \ZZ:j\mapsto \omega(k,j).$$
	\end{definition}
	
	\begin{lemma}\label{prop:kernelInjectsIntoDual:lemma}
		Let $I$ be the image of $\psi_z-\id$ and let $J<G/Z(G)$ be transversal to $I$. Let $K$ be the kernel of $\psi_z-\id$. Then the map $$K\rightarrow J^*:k\mapsto \omega_k(\_)$$
		is an injective linear map to the dual $J^*=\hom_\ZZ(J,\ZZ)$
	\end{lemma}
	\begin{proof}
		Linearity of this map is obvious.
		To show the map is injective, we need to show that the kernel is trivial, or thus that for $k\in K$, if $\omega_k(\_):J\rightarrow \ZZ:j\mapsto \omega(k,j)$ is trivial, then $k=0$. Notice that $\omega_k$ can be extended to a map with domain $G/Z$. On this domain, $\omega$ is non-degenerate, and thus is $\omega_k=0$ if and only if $k=0$. We will show that $\omega_k$ vanishes on $I$. Indeed every element in $I$ can be expressed as $\psi_z(w)-w$ where $w\in G/Z$. By bilinearity we have that $\omega(k,\psi_z(w)-w)=\omega(k,\psi_z(w))-\omega(k,w)$. As $k$ lies in the kernel of $\psi_z-\id$, the first term of the right hand side equal to $\omega(\psi_z(k),\psi_z(w))$. Furthermore as $\omega$ is $\psi_z$ invariant, we have that the second term also equals $\omega(\psi_z(k),\psi_z(w))$. It thus follows that $\omega_k$ vanishes on $I$.
		
		Suppose now that $\omega_k$ vanishes on $J$, then $\omega_k$ vanishes on the finite index subgroup $I+J$ of $G/Z$. As the codomain of $\omega$ is torsion free, it follows that $\omega_k$ vanishes on $M$ and thus, as $\omega$ is non-degenerate, that $k=0$.
		
	\end{proof}
	\begin{corollary}\label{prop:KMapsDim}
		Let $I$ be the image of $\psi_c-\id$ and let $J<G/G_2$ be transversal to $I$. Let $K$ be the kernel of $\psi_c-\id$, then the set $\{2\omega_k-\psi'(k)\mid k\in K\}$ forms a submodule of $\mathrm{Aff}(J,\ZZ)$ of dimension $\mf d_\psi$. Furthermore, this set contains no constant functions except the $0$ function.
	\end{corollary}
	\begin{proof}
		Notice that $2\omega_k-\psi'(k)$ is just trivial whenever $k\in Z(G)/G_2$, indeed in this case $\omega_k$ vanishes and thus by the degeneracy assumption, $\psi'(k)$ vanishes as well. The rank of this space is thus at most $\dim(K)-\dim(K\cap Z(G)/G_2)=(d_c-r_c)-(d_{z/c}-r_{z-c})$ dimensional. Conversely, if $k,k'\in K/ Z$ distinct, then by \cref{prop:kernelInjectsIntoDual:lemma}, we have that $2\omega_k-\psi'(k)$ and $2\omega_{k'}-\psi'(k')$ are distinct, affine maps. The rank of this space must thus also be at least $\dim(K/Z(G))=(d_c-r_c)-(d_{z/c}-r_{z-c})$. Furthermore by \cref{prop:kernelInjectsIntoDual:lemma}, if $2\omega_k-\psi'(k)$ is constant, then $k\in \ZZ$, and thus by degeneracy, $2\omega_k-\psi'(k)=0$.
	\end{proof}

	\begin{lemma}\label{prop:conjNumberOneForms:lemma}
		Let $I$ be the image of $\psi_c-\id$ and let $J<G/G_2$ be a transversal of $I$. For any $a\in G/G_2$, there exist some basis $\set{v_1,v_2,\cdots v_{d_c-r_c}}$ of $J$, together with some non-constant affine maps $\theta_i:\ZZ\rightarrow\ZZ:\lambda\mapsto x_i\lambda+y_i$ and a constant $D\in\NN_{>0}$, such that if $$\ggd(\theta_1(\lambda_1),\theta_2(\lambda_2),\cdots,\theta_{\mf d_\psi}(\lambda_{\mf d_\psi}))\mid c-c'$$ then the following are $\psi$-conjugate:$$(a+\sum_{j=1}^{d_c-r_c}\lambda_jv_j,c'),$$ $$(a+\sum_{j=1}^{d_c-r_c}\lambda_jv_j,c).$$ Conversely if these two are $\psi$-conjugate then $$\ggd(\theta_1(\lambda_1),\theta_2(\lambda_2),\cdots,\theta_{\mf d_\psi}(\lambda_{\mf d_\psi}))\mid D(c-c').$$
	\end{lemma}
	\begin{proof}
				
		Let $K$ be the kernel of $\psi_c-\id$ and let $\mf K<J^*$ be the subspace $\{\omega_k\mid k\in K\}$. Let $J_0<J$ be the common kernel of $\mf K$. By \cref{prop:KMapsDim}, we have that the dimension of $\mf K$ is precisely $\mf d_\psi$, and thus that the dimension of $J_0$ is precisely $d_c-r_c-\mf 
		d_\psi=d_{z/c}-r_{z/c}$.
		Notice that $J_0$ is isolated in $J$, we can thus find another submodule $J_1<J$ such that $J=J_0\oplus J_1$.
		Fix $\{v_1,\cdots, v_{\mf d_\psi}\}$ a basis of $J_1$ and $\{v_{\mf d_\psi+1},\cdots v_{d_c-r_c}\}$ a basis of $J_0$.

		$\set{v_1,v_2,\cdots v_{\mf d_\psi}}$ also induces a dual basis $\set{v_1^*,v_2^*,\cdots v_{\mf d_\psi}^*}$ of $J_1^*$. As $\mf K$ acts trivially on $J_0$, we can see $\mf K$ as a subspace of $J_1^*$. Notice that $\mf K$ and $J_1^*$ have the same dimension. It follows that $\mf K$ is a finite-index submodule of $J_1^*$. In particular we may find $\delta_i>0$ such that $\delta_iv_i^*\in 2\mf K$. Now the subspace generated by elements of the form $\delta_iv_i^*$  is finite index in $2\mf K$, call this index $D$. As $\delta_iv_i^*$ lies in $2\mf K$, there exist some $k_i\in \mf K$ such that $2\omega_{k_i}=\delta_iv_i^*$, define then $$\theta_i:\ZZ\rightarrow\ZZ:\lambda\mapsto2\omega(k_i,a+\lambda v_i)-\psi'(k_i).$$
		Notice that as $2\omega_{k_i}=\delta_iv_i^*$, the map $\theta_i$ is non-constant. We will demonstrate that the result holds for the above choices of $\set{v_1,v_2,\cdots v_{d_c-r_c}}$, $\{\theta_1,\cdots,\theta_{\mf d_\psi}\}$ and $D$.
		
		First suppose $\ggd(\theta_1(\lambda_1),\theta_2(\lambda_2),\cdots,\theta_{\mf d_\psi}(\lambda_{\mf d_\psi}))\mid c-c'$. By Bachet-Bézout, there exist constants $w_i\in\ZZ$ such that $c-c'=\sum w_i\theta_i(\lambda_i)$. Using the definition of $\theta_i$, it follows that $c-c'$ can be written as$$
		c-c'=\sum_{i=1}^{\mf d_\psi} w_i(2\omega(k_i,a+\lambda_iv_i)-\psi'(k_i)).
		$$
		Remember that $2\omega_{k_i}=\delta_iv_i^*$, it follows that whenever $j\neq i$, that $\omega(k_i,v_j)=0$. Using this, we have the following:$$
		c-c'=\sum_{i=1}^{\mf d_\psi} w_i(2\omega(k_i,a+\sum_{j=1}^{d_c-r_c}\lambda_jv_j)+\psi'(k_i)),
		$$
		which using linearity of $\psi'$ and bilinearity of $\omega$, can be rewritten as $$
		c-c'=2\omega(\sum_{i=1}^{\mf d_\psi} w_i k_i,a+\sum_{j=1}^{d_c-r_c}\lambda_jv_j)-\psi(\sum_{i=1}^{\mf d_\psi} w_i k_i).
		$$
		By Lemma \ref{prop:centerDifferenceInOmega:lemma}, it follows that $(a+\sum_{j=1}^{d_c-r_c}\lambda_jv_j,c)$ and $(a+\sum_{j=1}^{d_c-r_c}\lambda_jv_j,c')$ are $\psi$-conjugate.
		
		For the converse, assume that $(a+\sum_{j=1}^{d_c-r_c}\lambda_jv_j,c)$ and $(a+\sum_{j=1}^{d_c-r_c}\lambda_jv_j,c')$ are $\psi$-conjugate. Again by \cref{prop:centerDifferenceInOmega:lemma} there exists some $k\in K$ such that 
		$$
		c-c'=2\omega(k,a+\sum_{j=1}^{d_c-r_c}\lambda_jv_j)-\psi'(k).
		$$
		We may multiply the left and right hand side with $D$:$$
		D(c-c')=2\omega(Dk,a+\sum_{j=1}^{d_c-r_c}\lambda_jv_j)-\psi'(Dk).
		$$
		As $D$ is the index of $\langle\delta_1v_1^*,\cdots\delta_{\mf d_\psi}v_{\mf d_\psi}^*\rangle$ in $2\mf K$, there exist constants $w_i\in\ZZ$ such that $\omega_{Dk}=\sum w_i\delta_iv_i^*$. Or thus that $Dk=k_0+\sum_{i=1}^{\mf d} w_i k_i$, where $k_0\in K\cap Z(G)/G_2$.
		Using that $\psi'(k_0)=0$, this allows us to rewrite the above as:$$
		D(c-c')=\sum_{i=1}^{\mf d_\psi} w_i(2\omega(k_i,a+\sum_{j=1}^{d_c-r_c}\lambda_jv_j )-\psi'(k_i)).
		$$
		As $2\omega_{k_i}=\delta_iv_i^*$, we have that for $j\neq i$ that $2\omega(k_i,v_j)=0$. Allowing us to rewrite the above as $$
		D(c-c')=\sum_{i=1}^{\mf d_\psi} w_i (2\omega(k_i,a+\lambda_i v_i))-\psi(k_i)=\sum_{i=1}^{\mf d_\psi} w_i\theta_i(\lambda_i).
		$$
		It follows that $D(c-c')$ is divided by the greatest common divisor of the integers $\theta_i(\lambda_i)$
	\end{proof}

\section{Number Theoretic Estimates}\label{sec:numbertheory}

In this section, we estimate the following series.

\begin{lemma}\label{prop:sumofggd:lemma}
	Let $0<\mf d\leq l$ be integers.
	Let $\theta_i$ be fixed maps of the form $\theta_i:\ZZ\rightarrow\ZZ:x\mapsto a_ix+b_i$ where $a_i\neq 0$ and $b_i$ are integers and $i\in\range{\mf d}$, then $$
	\sum_{\lambda_1=-N}^N\sum_{\lambda_2=-N}^N\cdots\sum_{\lambda_{l}=-N}^N \ggd(\theta_1(\lambda_1),\theta_2(\lambda_2),\cdots,\theta_{\mf d}(\lambda_{\mf d}))\simeq N^ld_\mf d(N)
	$$
	where $d_\mf d$ is given by:$$
	d_\mf d:\NN_{>0}\rightarrow\RR^+_0:N\mapsto\begin{cases}
		N &\mf d=1\\
		\ln(N)&\mf d=2\\
		1&\mf d\geq 3.
	\end{cases}
	$$
\end{lemma}

One function that that pops up naturally while doing these manipulations is \emph{Euler's Totient function} which we will denote by $\varphi_e$. Remember that for a strictly positive natural number $n$, $\varphi_e(n)$ denotes the number of integers in $[1,n]$ that are relatively prime to $n$.
It is a well known fact that the following equality holds:$$
n=\sum_{d\mid n}\varphi_e(d).
$$
Also well known is that$$
\varphi_e(p^n)=(p-1)p^{n-1}
$$
whenever $p$ is prime and $n>1$, and that for $m$ and $n$ coprime we have $$
\varphi_e(mn)=\varphi_e(m)\varphi_e(n).
$$
We will mostly be interested in the assymptotic behaviour of the series
$$
S_{\mf d,m}(n)=\sum_{\substack{0<i<n\\\gcd(i,m)=1}}\frac{\varphi_e(i)}{i^\mf d}
$$
where $\mf d$ and $m$ are fixed positive integers.
We first resolve the case where $\mf d=2$.
\begin{lemma}\label{prop:seriesEstimate2:lemma}
	Let $m$ be a positive integer, then $$S_{2,m}(n)\simeq \ln (n).$$
\end{lemma}
\begin{proof}
	In case $m=1$, this is a well known result that follows for instance from Abel's summation formula, combined with the fact that $$\sum_{i\leq n}\varphi_e(i)= \frac{3n^2}{\pi^2}+ O(n\log n).
	$$
	We prove here that varying $m$ only changes the result up to ``$\simeq$''.
	First notice that $$S_{2,m}(n)\leq S_{2,1}(n).$$ We thus need to show that the other inequality holds up to a constant factor.
	Let $(p_i)$ be the sequence of primes, and let $\chi_{\leq n}$ be the indicator function, returning $1$ on an input less or equal to $n$, and returning $0$ otherwise. Factoring $i$, we may rewrite $S_{2,m}(n)$ as:$$
	\sum_{\substack{(e_j)\in \ZZ_{\geq 0}^\infty\\ p_j\mid m\Rightarrow e_j=0}}    \frac{\varphi_e(\prod p_j^{e_j})}{\prod p_j^{2e_j}}\chi_{\leq n}(\prod p_j^{e_j}).
	$$
	Using the definition of $\varphi_e$, we may again rewrite this as
	$$
	\sum_{\substack{(e_j)\in \NN^\infty\\ p_j\mid m\Rightarrow e_j=0}}    \prod_{\substack{j\in\NN\\e_j\geq 1}}\frac{p_j-1}{p_j^{e_j+1}}\chi_{\leq n}(\prod p_j^{e_j}).
	$$
	
	For $p$ a fixed prime number, notice that $$\sum_{e_p\in\NN}\frac{p-1}{p^{e_p+1}}$$ is a geometric series and as such converges to a non-zero constant $c_p$,
	let $c_m=\prod_{p\mid m} c_p$. We may multiply the original expression with $\frac{c_m}{c_m}$ without changing the result. Doing this returns $$
	\frac{1}{c_m} \left(\sum_{\substack{(e'_j)\in\NN^\infty\\p_j\nmid m\Rightarrow e'_j=0}}\prod_{\substack{j\in\NN\\e'_j\geq 1}}\frac{p_j-1}{p_j^{e'_j+1}}\right)\left(\sum_{\substack{(e_j)\in \NN^\infty\\ p_j\mid m\Rightarrow e_j=0}}    \prod_{e_j\geq 1}\frac{p_j-1}{p_j^{e_j+1}}\chi_{\leq n}\left(\prod_{e_j\geq 1} p_j^{e_j}\right)\right).$$
	Merging these two sums into one\footnote[2]{
		This is valid as the left sum is convergent, and the right sum only takes finitely many non-$0$ values.
	}, we obtain$$
	S_{2,m}(n)=\frac{1}{c_m}\sum_{\substack{(e_j)\in \NN^\infty}}\prod_{\substack{j\in\NN\\e_j\geq 1}}\frac{p_j-1}{p_j^{e_j+1}}\chi_{\leq n}\left(\prod_{\substack{j\in\NN\\\ggd(p_j,m)=1}}p_j^{e_i}\right).
	$$
	If we increase the argument of the function $\chi_{\leq n}$, then the total must decrease, we thus obtain that$$
	S_{2,m}(n)\geq \frac{1}{c_m}\sum_{\substack{(e_j)\in \NN^\infty}}\prod_{\substack{j\in\NN\\e_j\geq 1}}\frac{p_j-1}{p_j^{e_j+1}}\chi_{\leq n}\left(\prod_{\substack{j\in\NN\\\ggd(p_j,1)=1}}p_j^{e_j}\right).
	$$
	We thus have that $$
	S_{2,m}(n)\geq \frac{1}{c_m}\sum_{\substack{i<n\\\gcd(i,1)=1}}\frac{\varphi_e(i)}{i^2}=\frac{1}{c_m}S_{2,1}(n).
	$$
\end{proof}
The above works when the the exponent in the denominator is $2$. If the exponent is larger, then we have the following instead.
\begin{lemma}\label{prop:seriesEstimate3:lemma}
	Let $\mf d>2$ and let $m>0$ be a positive integer, then
	$$S_{\mf d,m}(n)$$ is bounded above by a constant.
\end{lemma}
\begin{proof}
	Notice that $\varphi_e(i)\leq i.$ The given function is thus bounded above by $\sum_{i\leq \NN}\frac{1}{i^{k-1}}$. This is a hyperharmonic series, and thus a convergent series.
\end{proof}

\begin{proof}[proof of \cref{prop:sumofggd:lemma}]
	Notice that $$
	\sum_{\lambda_{\mf d+1}=-N}^N\cdots\sum_{\lambda_{l}=-N}^N \ggd(\theta_1(\lambda_1),\theta_2(\lambda_2),\cdots,\theta_{\mf d}(\lambda_{\mf d}))
	$$
	is precisely equal to $$(2N+1)^{l-\mf d}\ggd(\theta_1(\lambda_1),\theta_2(\lambda_2),\cdots,\theta_{\mf d}(\lambda_{\mf d})).$$
	We may factor out this $(2N+1)^{l-\mf d}$, allowing us to reduce to the case where $l=\mf d$.

	First we demonstrate an upper bound.
	Let $N'$ be an upper bound in absolute value, of the values $\theta_i$ take on $[-N,N]$.  Then we can bound above sum above by:\begin{equation}\label{eq:sumofggd}
		\sum_{\lambda_1=-N'}^{N'}\sum_{\lambda_2=-N'}^{N'}\cdots\sum_{\lambda_{\mf d}=-N'}^{N'} \ggd(\lambda_1,\lambda_2,\cdots,\lambda_{\mf d}).
	\end{equation}
	The estimate of the above sum depends on $\mf d$: if $\mf d=1$, then this becomes $\sum_{\lambda=-N'}^{N'}\abs{\lambda}$ which is bounded above by $N'(N'+1)$. As $N'$ is linear in $N$, we obtain in this case that the expression is, up to a factor, bounded by $N^{2}$.
	Now assume that $\mf d\geq 2$. We show that this sum is bounded above by the following:$$
	\sum_{i\in\NN_{>0}} \left(\left(2\left\lfloor\frac{N'}{i}\right\rfloor+1\right)^{\mf d}-1\right)\varphi_e(i).
	$$
	
	Indeed, notice that for $j\geq 1$, $j=\sum_{i\mid j}\varphi_e(i)$. Thus the sum (\ref{eq:sumofggd}), is thus equal to the sum over all possible integers $i$, of $\varphi_e(i)$ times the number of times in the given range $i$ divides $\ggd(\lambda_1,\lambda_2,\cdots,\lambda_{\mf d})$. This last condition is fulfilled precisely when $i$ divides both $\lambda_1,\lambda_2,\cdots,\lambda_{\mf d}$, and this thus happens precisely $\left(\left(2\lfloor\frac{N'}{i}\rfloor+1\right)^{\mf d}-1\right)$ times (where the $-1$ accounts for ${\lambda_i=0}$).
	We thus have that (\ref{eq:sumofggd}) is bounded above by $$\sum_{i\in\NN_{>0}} \left(\left(2\left\lfloor\frac{N'}{i}\right\rfloor+1\right)^{\mf d}-1\right)\varphi_e(i).$$
	To estimate this from above, notice that when $i>N'$, then the corresponding argument vanishes. We thus have an upper bound$$
	\sum_{i\leq N'} (2\frac{N'}{i}+1)^\mf d\varphi_e(i)
	$$
	which is in turn bounded above by$$
	(4N')^\mf d S_{\mf d,1}(N').
	$$
	By \Cref{prop:seriesEstimate3:lemma} or by \Cref{prop:seriesEstimate2:lemma}, this is bounded above, up to ``$\prec$'', by $N'^2\ln(N')$ when $\mf d=2$ or by $N'^{\mf d}$ when $\mf d\geq 3$. The upper bound follows as $N'$ is at most linear in $N$.
	
	Now to demonstrate the lower bound:
	
	Notice that the functions $\theta_i$ are non-constant and affine, and thus vanish for at most $1$ value of $\lambda_i$.
	If $\mf d\geq 3$, we may estimate from below by replacing the summands of \ref{eq:sumofggd} by $1$ if it is larger than $1$, and retain the $0$ otherwise. This sum thus is bounded below by $N^{\mf d}-1\simeq N^{\mf d}$. In this case, we thus have the desired lower bound.
	
	When $\mf d=1$, we obtain the following sum$$
	\sum_{\lambda=-N}^N\abs{\theta_1(\lambda)}.
	$$
	This is bounded below by $\frac{N^2}{2}$
	
	Only the case where $\mf d=2$ remains.
	
	Without loss of generality assume $\theta_i$ only vanishes for a negative value of $\lambda_i$, replacing $\theta_i(\lambda_i)$ by $\theta_i(-\lambda_i)$ if necessary. We estimate the sum from below by$$
	\sum_{\lambda_1=1}^N\sum_{\lambda_{2}=1}^N \ggd(\theta_1(\lambda_1),\theta_2(\lambda_2)).
	$$
	
	For $j$ a strictly positive integer, we again have $$j=\sum_{\substack{k\in\NN_{>0}\\k\mid j}}\varphi_e(k).$$
	We may thus rewrite the previous sum as$$
	\sum_{k\in\NN_{>0}}\varphi_e(k) \sum_{\lambda_1=1}^N\sum_{\lambda_{2}=1}^N \chi_{k\mid}(\ggd(\theta_1(\lambda_1),\theta_2(\lambda_2)))
	$$
	where $\chi_{k\mid}(x)$ is the indicator function returning $1$ if $k\mid x$ and $0$ otherwise.
	To compute this sum, we count for a given integer $k$, the number of integers $j\in[0,N]$ such that $k\mid \theta_i(j)$. 
	Assume $\ggd(k,a_i)=1$, then $a_i$ is a generator of $\frac{\ZZ}{k\ZZ}$ and thus for some value of $j\in[1,k]$, $k$ must divide $\theta_i(j)$. Similarly this holds for some value $j\in[k+1,2k]$, for some value $j\in[2k+1,3k]$,$\cdots$.
	Whenever $\ggd(k,a_i)=1$, there are thus at least $\lfloor\frac{N}{k}\rfloor$ choices of $j\in[1,N]$ such that this happens.
	Let $a$ now be the least common multiple of the slopes $\set{a_1,a_2}$ and let $\ggd(k,a)=1$, then there are at least $\lfloor\frac{N}{k}\rfloor^2$ choices of pairs $\lambda_1,\lambda_2$, such that $k$ divides both $\theta_1(\lambda_1)$ and $\theta_2(\lambda_2)$.
	The sum from before, can thus be estimated from below by$$
	\sum_{\substack{k\in\NN_{>0}\\\ggd(k,m)=1}}\left\lfloor\frac{N}{k}\right\rfloor^2\varphi_e(k).
	$$
	If $k>N$, then the summand vanishes, on the other hand, if $k\leq N$, then $\floor{\frac{N}{k}}\geq \frac{N}{2k}$, we thus obtain a final lower bound of$$
	\frac{1}{4}N^2S_{2,m}(N),
	$$
	which by Lemma \ref{prop:seriesEstimate2:lemma} is equivalent to $N^2\ln(N)$.
\end{proof}

\section{Proof of the degenerate case}\label{sec:degen}

	With this we complete the proof of \Cref{prop:completemaintheorem}.
		
	\begin{customthm}{B.2}\label{prop:maintheoremdegen}
		\sloppy
		Let $G$ and $\psi$ be as before, suppose that $\psi\vert_{G_2}$ is the identity and that ${\psi'(\ker(\psi_{z/c}-\id))}$ is trivial, then $$\gr_{\psi,G}(n)\simeq n^{d_c-r_c}d_{\psi}(n)$$ where $d_\psi$ is given by$$\begin{cases}
			\text{$d(n)=1$ if $\mf d_\psi\geq 3$;}\\
			\text{$d(n)=\log(n)$ if $\mf d_\psi=2$;}\\
			\text{$d(n)=n$ if $\mf d_\psi=1$;}\\
			\text{$d(n)=n^2$ if $\mf d_\psi=0$.}
		\end{cases}$$
	\end{customthm}	
	\begin{proof}
		Fix a generating set $S$ on $G$, containing a generator of $G_2$. 
		As before let $I$ be the image of $\psi_c-\id$ and let $J,\mf A$ be such that $(I,J,\mf A)_{G/G_2}$ is a transversality triple.
		We start by giving the upper bound.
		Let $n>0$ be arbitrary.
		Let $g\in G$ such that $\norm{g}_S\leq n$. By \cref{prop:conjToCokernelStrong:lemma}, there exists some $g'\in G$ such that $\norm{g'}_S\leq C_1\norm{g}_S$, such that $[g']_\psi=[g]_\psi$ and such that $g'G_2\in\mf A+J$.
		We can thus bound $$\#\left\{[g]_\psi\big \vert \norm{g}_S\leq n\right\}$$ from above by $$
		\#{\left\{[g]_\psi\big \vert \norm{g}_S\leq C_1n, gG_2\in\mf A+J \right\}}.
		$$
		As we are only interested in $\gr_{\psi,G}(n)$ up to $\simeq$, it thus suffices to bound$$
		\#{\left\{[g]_\psi\big\vert \norm{g}_S\leq n, gG_2\in\mf A+J \right\}}.
		$$
		As $\mf A$ is finite, we can do these estimates independently for each of the elements of $\mf A$.
		Let $a\in A$ be arbitrary. Let $\{v_1,\cdots,v_{d_c-r_c}\}$ and $\{\theta_1,\cdots,\theta_{\mf d_\psi}\}$ be such as in \cref{prop:conjNumberOneForms:lemma}.
		
		We split the set  $$
		X_a(n)=\left\{[g]_\psi\big\vert \norm{g}_S\leq n, gG_2\in\mf a+J \right\}
		$$
		into two.
		Let $$
		X_0(n)=\left\{[g]_\psi\in X_a\Big\vert\exists \lambda_i: gG_2=a+\sum_{i=1}^{d_c-r_c}\lambda_iv_i\text{ with }\theta_i(\lambda_i)=0\text{ for }i\in\range{\mf d_\psi}\right\}.
		$$
		Now let $X_1(n)=X_a(n)\backslash X_0(n)$, in particular there exists for each of these classes some $i\in \range{\mf d_{\psi}}$ such that $\theta_i(\lambda_i)\neq 0$.
		Notice that it might be possible for one of these sets to be empty.
		
		First we prove that there exists some constant $C_2$, independent of $n$, such that $\#{X_0(n)}\leq C_2n^{d_c-r_c-\mf d_\psi+2}$. 
		Notice that as $\theta_i$ are all non-constant, for every $i\in\range{\mf d_{\psi}}$, there exists at most one value $\lambda$, such that $\theta_i(\lambda)=0$.
		There exists some constants $C_3,C_4$, such that if $$(g_1,g_2)\in X_0(n),$$ then $\norm{g_1}_{S_c}\leq C_3n$ and $\abs{g_2}\leq C_4n^2$. And thus, after enlarging $C_3$ if necessary, such that $\lambda_i\leq C_3n+\norm{a}_{S_c}$ for all $i\in [\mf d_\psi+1,d_c-r_c]$. We thus have that $\#{X_0(n)}$ is at most ${C_4n^2(C_3n+\norm{a}_S)^{d_c-d_r-\mf d_\psi}}$. If we now take $C_5=\max(C_4,C_3+\norm{a}_S)$, then we finally obtain an upper bound of $(C_5n)^{d_c-r_c-\mf d_\psi+2}$.
			
		Now we proceed with an upper bound for the size of $X_1(n)$.
		As before, we can find some $C_3$ such that if $(g_1,g_2)=(a+\sum{\lambda_iv_i},g_2)\in X_1(n)$, then $\abs{\lambda_i}\leq C_3n$. By assumption, there exists some $i$, in $\range{\mf d_\psi}$ such that $\theta_i(\lambda_i)\neq 0$ and thus is $\gcd(\theta_1(\lambda_1),\cdots, \theta_{\mf d_\psi}(\lambda_{\mf g_\psi}))\neq 0$. By \cref{prop:conjNumberOneForms:lemma}, we can thus estimate 
		$$\#{X_1(n)}\leq
		\sum_{\lambda_1=-N}^N\sum_{\lambda_2=-N}^N\cdots\sum_{\lambda_{d_c-r_c}=-N}^N \ggd(\theta_1(\lambda_1),\theta_2(\lambda_2),\cdots,\theta_{\mf d_\psi}(\lambda_{\mf d_\psi}))$$
		where $N=C_3n$. By \cref{prop:sumofggd:lemma} this is bounded above by ${C_6N}^{d_c-r_c}d_{\psi}(C_6N)$. for some constant $C_6>0$.
			
		Adding these functions together over all possible choices of $a\in \mf A$, we have that 
		$$\gr_{\psi,G}(n)\prec\sum_{a\in\mf A}\#{X_a(n)}\prec\max(n^{d_c-r_c-\mf d_\psi+2},n^{d_c-r_c}d_{\mf d_\psi}).$$
		Which one of these arguments is larger depends on $\mf d_\psi$.
		If $\mf d_\psi=0$, then the second vanishes and we are left with $n^{d_c-d_r+2}$. In case $\mf d_\psi=1$, Then in both arguments we have $n^{d_c-d_r+1}$. In case $\mf d_\psi\geq2$, then the second term dominates the first and we are left with $n^{d_c-r_c}d_{\mf d_\psi}(n)$. In all these cases, the desired upper bound holds.

		Now we proceed with the lower bound. Let $(g_1,g_2)$,$(g_1',g_2')\in G$ be arbitrary such that $g_1,g_2\in J$. By \cref{prop:conjDifferByI:lemma} it follows that if $(g_1,g_2)\sim_\psi(g_1',g_2')$, then $g_1=g_1'$. 
		First we proceed with the lower bound in case $\mf d_\psi=0$.
		In this case, from combining the previous with \cref{prop:conjNumberOneForms:lemma} the elements of the form$$
		\sum_{(\lambda_i)_i\in\ZZ^{d_c-r_c}}(\lambda_iv_i,g_2)
		$$
		all lie in distinct $\psi$-classes. Let $C_7=\max(\norm{v_i})$. For each $v_i$, pick $w_i\in G$ with $w_iG_2=v_i$ such that $\norm{w_i}=\norm{v_i}$ Pick $s_1,s_2\in S$ such that $[s_1,s_2]$ is non-trivial. Then for $\lambda_i\in\range{n}, \kappa_1,\kappa_2\in \range{n}$, consider the following elements:$$
		\prod_{i=1}^{d_c-r_c}w_i^{\lambda_i}\: [s_1^n,s_2^{\kappa_1}][s_1,s_2^{\kappa_2}].
		$$
		These elements, by construction, all have a word norm of at most $$
		(d_c-r_c)C_7n + 4n+2n+2\leq ((d_c-r_c)C_7+8)n.
		$$
		As all these elements are distinct, then we have $\gr_{\psi,G}^S(((d_c-r_c)C_7+8)n)\geq n^{d_c-r_c+2}$.
		
		Now assume that $\mf d_\psi\geq 1$.
		For $\lambda_i\in [-n,n]$ and for $0\leq g_2<\gcd(\theta_1(\lambda_1),\cdots,\theta_{\mf d_\psi}(\lambda_{\mf d_\psi}))$ with $g_2$ even, we have by \cref{prop:conjNumberOneForms:lemma} that all the elements
		$$
		\prod_{i=1}^{d_c-r_c}w_i^{\lambda_i}\:(0,g_2)
		$$
		lie in distinct $\psi$-conjugacy classes.
		These elements are also all of norm at most $C_7(d_c-r_c)n+n$.
		By \cref{prop:sumofggd:lemma} we thus obtain that $$\gr_{\psi,G}^{S}(((d_c-r_c)C_7+1)n)\succ n^{d_c-r_c}d_{\mf d}(n).$$
		In both cases we thus have the desired lower bound.		
	\end{proof}

	\bibliographystyle{alpha}
	\bibliography{heisenberg}
\end{document}